\documentclass[letterpaper, 10 pt, conference]{ieeeconf}
\IEEEoverridecommandlockouts

\usepackage[utf8]{inputenc} 
\usepackage[T1]{fontenc}    
\usepackage{url}            
\usepackage{hyperref}       

\usepackage{booktabs}       
\usepackage{amsfonts}       
\usepackage{nicefrac}       
\usepackage{microtype}      

\usepackage{amsmath}
\usepackage{amssymb}
\usepackage{graphicx}
\usepackage{subfig}
\usepackage{textcomp}
\usepackage{xcolor}
\usepackage{epstopdf}
\usepackage{algorithm}
\usepackage{algorithmic}
\usepackage{balance}

\def \hy #1{\textcolor{black}{#1}}

\newtheorem{Thm}{Theorem}

\newtheorem{Lem}{Lemma}
\newtheorem{Ass}{Assumption}
\newtheorem{Rem}{Remark}

\title{\LARGE \bf  
An Optimistic Gradient Tracking Method for Distributed Minimax Optimization
}

\author{Yan Huang, Jinming Xu, Jiming Chen, and Karl Henrik Johansson 
\thanks{Yan Huang and Karl Henrik Johansson are with the Division of Decision and Control Systems, School of EECS, KTH Royal Institute of Technology, SE-100 44 Stockholm, Sweden. Email: \{yahuang, kallej\}@kth.se}
\thanks{Jinming Xu and Jiming Chen are with the College of Control Science and Engineering, Zhejiang University, Hangzhou 310027, China. Email: \{jimmyxu, cjm\}@zju.edu.cn
}}

\begin{document}

\maketitle

\begin{abstract}
This paper studies the distributed minimax optimization problem over networks. To enhance convergence performance, we propose a distributed optimistic gradient tracking method, termed DOGT, which solves a surrogate function that captures the similarity between local objective functions to approximate a centralized optimistic approach locally. Leveraging a Lyapunov-based analysis, we prove that DOGT achieves linear convergence to the optimal solution for strongly convex-strongly concave objective functions while remaining robust to the heterogeneity among them. Moreover, by integrating an accelerated consensus protocol, the accelerated DOGT (ADOGT) algorithm achieves an optimal convergence rate of  $\mathcal{O} \left( \kappa \log \left( \epsilon ^{-1} \right) \right) $ and communication complexity of $\mathcal{O} \left( \kappa \log \left( \epsilon ^{-1} \right) /\sqrt{1-\sqrt{\rho _W}} \right)$ for a suboptimality level of $\epsilon>0$, where $\kappa$ is the condition number of the objective function and $\rho_W$ is the spectrum gap of the network. Numerical experiments illustrate the effectiveness of the proposed algorithms.
\end{abstract}

\section{Introduction}

Minimax optimization has gained significant attention over the past decade due to its broad applications in robust optimization \cite{mohri2019agnostic, sinha2017certifying}, game theory \cite{vamvoudakis2023game}, and generative adversarial networks (GANs) \cite{goodfellow2014generative, gulrajani2017improved}, among others.
With the increasing scale of data and devices \cite{konevcny2016federated}, distributed optimization has emerged as a key approach in large-scale optimization \cite{nedic2018distributed}, machine learning \cite{lian2017can}, and control \cite{molzahn2017survey}.

In this paper, we consider the distributed minimax optimization problem jointly solved by a network of $n$ nodes:
\begin{equation}\label{Prob_minimax}
\underset{x\in \mathbb{R} ^p}{\min}\,\,\underset{y\in \mathbb{R} ^d}{\max}f\left( x,y \right) :=\frac{1}{n}\sum_{i=1}^n{f_i\left( x,y \right)},
\end{equation}
where $f_i$, $i=1,\dots,n$, is the local objective function, $x\in \mathbb{R} ^p$ and $y\in \mathbb{R} ^d$ are primal and dual variables to be minimized and maximized, respectively.
\hy{This formulation is general, as it recovers the traditional distributed minimization problem \cite{nedic2018distributed} as a special case, thereby capturing a broader class of applications. For instance, in distributed training of Wasserstein GANs \cite{liu2020decentralized, huang2024achieving}, nodes collaborate to minimize the generator while simultaneously maximizing the discriminator, thus enabling data-parallel accelerated training.}
A straightforward approach to solving this problem is to apply centralized minimax algorithms, such as gradient descent ascent (GDA) \cite{dem1972numerical}, optimistic gradient descent ascent (OGDA) \cite{daskalakis2017training} and extragradient (EG) \cite{korpelevich1976extragradient}, in combination with the averaging consensus protocol \cite{xiao2006distributed}. 
For instance, Deng and Mahdavi \cite{deng2021local} proposed Local SGDA, which integrates FedAvg from federated learning \cite{mcmahan2017communication} with stochastic GDA to reduce communication costs. Chen et al. \cite{chen2024efficient} introduced a variance-reduction-based decentralized GDA method, achieving improved sample complexity under the stronger assumption of average smoothness for each sample. However, as shown in \cite{mokhtari2020unified}, centralized GDA methods do not achieve linear convergence for bilinear functions. We will show later that this limitation also applies to the distributed GDA (DGDA) algorithm.

OGDA and EG methods have gained popularity in recent literature due to their ability to achieve linear convergence rates for strongly convex-strongly concave and bilinear objective functions \cite{mokhtari2020unified, azizian2020tight}, as well as their effectiveness in training GANs \cite{daskalakis2017training, liang2019interaction}.
In decentralized settings, Liu et al. \cite{liu2020decentralized} proposed a distributed optimistic gradient descent ascent method, DPOSG, for training large-scale GANs and analyzed its convergence for nonconvex-nonconcave (NC-NC) objective functions.
Beznosikov et al. \cite{beznosikov2020distributed} introduced an EG-based FedAvg method and proved that with an accelerated consensus protocol, it achieves a near-optimal convergence rate, matching the lower bound for (strongly) convex-(strongly) concave objective functions up to logarithmic factors.
Liu et al. \cite{liu2019decentralized} developed a decentralized proximal point method for NC-NC problems, which requires solving a subproblem at each iteration.
However, the above methods establish convergence under assumptions such as uniformly bounded gradient norms \cite{liu2020decentralized} or bounded feasible domains \cite{liu2019decentralized, beznosikov2020distributed}. 
Otherwise, a decaying stepsize is required to remove a steady-state error \cite{cai2024diffusion}.
These assumptions restrict the applicability of their theoretical results in many real-world tasks due to their inability to handle the heterogeneity of local objective functions \cite{lian2017can}, i.e., the gradients of each $f_i$ differ, leading to distinct local optimas. This heterogeneity is recognized as a fundamental challenge in distributed optimization.

Gradient tracking (GT) methods \cite{xu2015augmented, pu2021distributed} are widely used in distributed minimization to mitigate the effects of data/function heterogeneity. Recent studies have explored their application to minimax optimization. For instance, Wai et al. \cite{wai2018multi} introduced a GDA-based GT method for multi-agent reinforcement learning and established a linear convergence rate for SC-SC settings.
Mukherjee et al. \cite{mukherjee2020decentralized} combined GT with EG methods, achieving linear convergence for strongly convex-strongly concave (SC-SC) objectives without assuming a bounded gradient norm or feasible domain. However, there remains a gap between the convergence rates in \cite{wai2018multi, mukherjee2020decentralized} and the established lower bound in \cite{beznosikov2020distributed}.
Another line of research utilizes second-order similarity based on the mirror-descent method \cite{sonata, beznosikov2021distributed, zhou2024near}, where the difference between the second-order derivative matrices of $f_i$ is bounded by a finite quantity $\delta$. This approach improves algorithmic efficiency when $\delta$ is smaller than the smoothness condition number. However, these methods often incur high computational costs due to the need to solve a subproblem.

To address these challenges, inspired by the mirror-descent-type method in decentralized settings \cite{sonata, beznosikov2021distributed}, we propose a distributed optimistic gradient tracking method, termed DOGT, for solving Problem \eqref{Prob_minimax}. This algorithm solves a surrogate objective function that captures the similarity between the local gradient and an estimator of the global gradient, using one-step approximate solutions to avoid the need for inner loops to solve subproblems. Theoretically, we prove that DOGT achieves a linear convergence rate and outperforms the result in \cite{mukherjee2020decentralized} by a factor of $\mathcal{O} \left( \kappa ^{1/3} \right) $ when the graph connectivity is strong. Furthermore, by integrating an accelerated consensus protocol with prior knowledge of the graph, we propose ADOGT and prove that it achieves an optimal convergence rate of $\mathcal{O} \left( \kappa \log \left( \epsilon ^{-1} \right) \right)$ and communication complexity of $\mathcal{O} \left( \kappa \log \left( \epsilon ^{-1} \right) /\sqrt{1-\sqrt{\rho _W}} \right)$, matching the lower bound established in \cite{beznosikov2020distributed} for deterministic settings under the assumptions considered in this work. Numerical results validate our theoretical findings.

\textbf{Paper organization.} The rest of the paper is structured as follows: Section \ref{Sec_prob_algorithm} formulates the distributed minimax optimization problem with several common assumptions and presents the algorithm design of DOGT and ADOGT. Section \ref{Sec_convergence} provides the main convergence results of the proposed algorithms. Section \ref{Sec_experiment} presents numerical experiments to validate our theoretical findings. Finally, Section \ref{Sec_conclusion} concludes the paper, and several supporting lemmas for the proof of the main results are provided in the appendix.

\textbf{Notations.} Throughout this paper, we adopt the following notations: $\left< \cdot ,\cdot \right> $ is the inner product of vectors, $\left\| \cdot \right\|$ represents the Frobenius norm, $\lceil \cdot \rceil$ indicates the ceiling operation, $\mathbf{1}$ represents the all-ones vector, $\mathbf{I}$ denotes the identity matrix, and $\mathbf{J}=\mathbf{1}\mathbf{1}^{\top}/n$ denotes the averaging matrix.

\section{Problem Formulation and Algorithm Design}\label{Sec_prob_algorithm}

\subsection{Distributed Minimax Optimization Problem}
To solve the distributed minimax optimization problem \eqref{Prob_minimax} in a decentralized manner, we consider each node $i\in[n]$ maintaining a local copy of the global primal and dual variables, i.e., $x_i\in \mathbb{R} ^p, y_i\in \mathbb{R} ^d$, and optimize the following constrained problem, which has the same optimal solution as the original problem \eqref{Prob_minimax}: 
\begin{equation}\label{Prob_minimax_constained}
\begin{aligned}
&\underset{\mathbf{x}\in \mathbb{R} ^{n\times p}}{\min}\,\,\underset{\mathbf{y}\in \mathbb{R} ^{n\times d}}{\max}F\left( \mathbf{x},\mathbf{y} \right) :=\frac{1}{n}\sum_{i=1}^n{f_i\left( x_i,y_i \right) ,}
\\
&\text{s.t.} \quad x_i=x_j,\,\, y_i=y_j, \,\, \forall i, j=1,\dots ,n
\end{aligned}
\end{equation}
where 
\begin{equation*}
\begin{aligned}
\mathbf{x}&:=\left[ x_1,x_2,\dots ,x_n \right] ^{\top}\in \mathbb{R} ^{n\times p},
\\
\mathbf{y}&:=\left[ y_1,y_2,\dots ,y_n \right] ^{\top}\in \mathbb{R} ^{n\times d}
\end{aligned}
\end{equation*}
are the collections of the primal and dual variables, respectively, which are required to reach consensus.
Moreover, the nodes are connected over a decentralized network, and its topology is modeled as an undirected graph $\mathcal{G}=(\mathcal{V},\mathcal{E})$, where $\mathcal{V}=\{1,2,...,n\}$ denotes the set of agents, and $\mathcal{E}\subseteq\mathcal{V}\times \mathcal{V}$ denotes the set of edges consisting of ordered pairs $(i,j)$ modeling the communication link from $j$ to $i$. Each node $i$ only communicates with its neighbor $\mathcal{N}_i=\left\{ j\,|\,j\ne i,\, \left( i, j \right) \in \mathcal{E} \right\} $, including $i$ itself.
Then, we make the following commonly used assumptions on the differentiable objective function $f_i$, its gradient $\nabla f_i$, and the graph $\mathcal{G}$.

\begin{Ass}[Convexity and Concavity]\label{Ass_conv_conc}
Each objective function $f_i\left( x, y \right) $ is $\mu$-strongly convex in $x$ and $\mu$-strongly concave in $y$, i.e., $\forall x, x^{\prime}\in \mathbb{R}^p$, $\forall y, y^{\prime}\in \mathbb{R}^d$ and $\mu>0$,
\begin{equation*}
\begin{aligned}
&f_i\left( x^{\prime},y \right) -f_i\left( x,y \right) 
\\
&\geqslant \left< \nabla _xf_i\left( x,y \right) ,x^{\prime}-x \right> +\frac{\mu}{2}\left\| x-x^{\prime} \right\| ^2,
\\
&f_i\left( x,y^{\prime} \right) -f_i\left( x,y \right) 
\\
&\leqslant \left< \nabla _yf_i\left( x,y \right) ,y^{\prime}-y \right> -\frac{\mu}{2}\left\| y-y^{\prime} \right\| ^2.
\end{aligned}
\end{equation*}
\end{Ass}

\begin{Ass}[Smoothness]\label{Ass_smoothness}
Each objective function $f_i\left( x, y \right) $ is jointly $L$-smooth in $x$ and $y$, i.e., $\forall x, x^{\prime} \in \mathbb{R}^p$ and $\forall y, y^{\prime} \in \mathbb{R}^d$, there exists a constant $L>0$ such that for $z\in \left\{ x,y \right\}$,
\begin{equation*}
\begin{aligned}
&\left\| \nabla_z f_i\left( x,y \right) -\nabla_z f_i\left( x^{\prime},y^{\prime} \right) \right\| ^2
\\
&\leqslant L^2\left( \left\| x-x^{\prime} \right\| ^2+\left\| y-y^{\prime} \right\| ^2 \right).
\end{aligned}
\end{equation*}
\end{Ass}

\begin{Ass}[Graph connectivity] \label{Ass_graph}
The weight matrix $W=\left[ w_{i,j} \right] _{i,j=1}^{n}$ induced by graph $\mathcal{G}$ is doubly stochastic, i.e., $W\mathbf{1}=\mathbf{1},\mathbf{1}^{\top}W=\mathbf{1}^{\top}$ and $\rho _W:= \left\| W-\mathbf{J} \right\|_2 ^2 <1$.
\end{Ass}

\subsection{Algorithm Design}

Motivated by the SONATA algorithm \cite{sonata} originally designed for minimization problems, we propose a distributed optimistic gradient tracking (DOGT) algorithm for solving the distributed minimax problem \eqref{Prob_minimax_constained}. In particular, we consider the following surrogate function for minimizing on $x_i\in\mathbb{R}^p$ and maximizing on $y_i\in\mathbb{R}^d$:
\begin{equation}\label{Prob_surrogate}
\begin{aligned}
f_i\left( x_i,y_i \right) &+\frac{1}{2\gamma}\left\| x_i-x_{i,k} \right\| ^2+\frac{1}{2\gamma}\left\| y_i-y_{i,k} \right\| ^2
\\
&+\underset{\mathrm{similarity \,\, in \,\,} x}{\underbrace{\left( p_{i,k}-\nabla _xf_i\left( x_{i,k},y_{i,k} \right) \right) }}^{\top}\left( x_i-x_{i,k} \right) 
\\
&+\underset{\mathrm{similarity \,\, in \,\,} y}{\underbrace{\left( q_{i,k}-\nabla _yf_i\left( x_{i,k},y_{i,k} \right) \right) }}^{\top}\left( y_i-y_{i,k} \right),
\end{aligned}
\end{equation}
where $p_{i,k}$ and $q_{i,k}$ are the gradient tracking variables used for estimating the global gradient at iteration $k$ with respect to $x$ and $y$, receptively. This surrogate function captures the difference between the local and global gradients. Solving it in a decentralized manner helps to asymptotically reduce this gap and approximate a centralized optimistic method, thereby improving the convergence \cite{sonata}.
Specifically, according to the first-order optimality condition of \eqref{Prob_surrogate}, and taking the primal variable $x$ as an example, we derive the following proximal point method:
\begin{equation}
\begin{aligned}
x_{i,k+1}=x_{i,k}&-\gamma \nabla _xf_i\left( x_{i,k+1}, y_{i,k+1} \right) 
\\
&-\gamma \left( p_{i,k}-\nabla_x f_i\left( x_{i,k},y_{i,k} \right) \right).
\end{aligned}
\end{equation}
To obtain an algorithm that is easy to deploy without requiring the exact solution of the surrogate function \eqref{Prob_surrogate}, we use the approximation 
\begin{equation}
\begin{aligned}
&\nabla _xf_i\left( x_{i,k+1},y_{i,k+1} \right) 
\\
&\approx 2\nabla _xf_i\left( x_{i,k},y_{i,k} \right) -\nabla _xf_i\left( x_{i,k-1},y_{i,k-1} \right),
\end{aligned}
\end{equation}
which has an error bounded by $o\left( \gamma^2 \right)$ \cite{mokhtari2020unified}. Based on this, we propose the DOGT algorithm, an optimistic gradient descent ascent method with gradient tracking. The pseudo-code for DOGT is provided in Algorithm \ref{Alg_DOST}. \hy{Notably, DOGT adopts a single-loop structure, making it more suitable for practical deployment than mirror-descent-based methods \cite{sonata, beznosikov2021distributed}, particularly in scenarios with data heterogeneity.}

\begin{algorithm}
\caption{\textbf{Distributed Optimistic Gradient Tracking (DOGT)}}
\begin{algorithmic}[1]\label{Alg_DOST}
\renewcommand{\algorithmicrequire}{\textbf{Initialization: }}
\REQUIRE {Initial points $x_{i,0}\in \mathbb{R}^p$, $y_{i,0}\in \mathbb{R}^d$, initial gradient tracking variables $\nabla _xf_{i,-1}=p_{i,0}=\nabla _xf_i\left( x_{i,0},y_{i,0} \right) $, $\nabla _yf_{i,-1}=q_{i,0}=\nabla _yf_i\left( x_{i,0},y_{i,0} \right) $, and stepsize $\gamma >0$.
}
\
\FOR{iteration $k = 0,1,\dots$, each node $i\in[n]$,}

\STATE Optimistic gradient descent ascent with gradient tracking variables:
\begin{equation*}
\begin{aligned}
x_{i,k+1}&=x_{i,k}-\gamma \left( {p_{i,k}}+\nabla _xf_{i,k}-\nabla _xf_{i,k-1} \right) ,
\\
y_{i,k+1}&=y_{i,k}+\gamma \left( {q_{i,k}}+\nabla _yf_{i,k}-\nabla _yf_{i,k-1} \right) .
\end{aligned}
\end{equation*}

\STATE {Compute gradient $\nabla _xf_{i,k+1}$ and $\nabla _yf_{i,k+1}$}.

\STATE Update gradient tracking variables:
\begin{equation*}
\begin{aligned}
p_{i,k+1}&={p_{i,k}}+\nabla _xf_{i,k+1}-\nabla _xf_{i,k},
\\
q_{i,{k+1}}&=q_{i,k}+\nabla _yf_{i,k+1}-\nabla _yf_{i,k}.
\end{aligned}
\end{equation*}

\STATE Inter-node communication for packaged message $\theta _{i,k+1}:=\left\{ x_{i,k+1}, y_{i,k+1}, p_{i,k+1}, q_{i,k+1} \right\} $:
\begin{equation*}
\theta _{i,k+1}\gets \sum_{j\in \mathcal{N} _i}{w_{i,j}\theta _{i,k+1}}.
\end{equation*}
\ENDFOR
\end{algorithmic}
\end{algorithm}

For brevity, we introduce the following notations for the gradients of all nodes at iteration $k$:
\begin{equation*}
\begin{aligned}
\nabla _xF_k&:=\left[ \dots ,\nabla _xf_i\left( x_{i,k},y_{i,k} \right) ,\dots \right] ^{\top} \in \mathbb{R} ^{n\times p},
\\
\nabla _yF_k&:=\left[ \dots ,\nabla _yf_i\left( x_{i,k},y_{i,k} \right) ,\dots \right] ^{\top} \in \mathbb{R} ^{n\times d}.
\end{aligned}
\end{equation*}
Then, the DOGT algorithm can be rewritten in the following compact form:
\begin{equation}\label{Eq_DOGT}
\begin{aligned}
\mathbf{x}_{k+1}&=W\left( \mathbf{x}_k-\gamma \left( \mathbf{p}_k+\nabla _xF_k-\nabla _xF_{k-1} \right) \right),
\\
\mathbf{y}_{k+1}&=W\left( \mathbf{y}_k+\gamma \left( \mathbf{q}_k+\nabla _yF_k-\nabla _yF_{k-1} \right) \right),
\\
\mathbf{p}_{k+1}&=W\left( \mathbf{p}_k+\nabla _xF_{k+1}-\nabla _xF_k \right) ,
\\
\mathbf{q}_{k+1}&=W\left( \mathbf{q}_k+\nabla _yF_{k+1}-\nabla _yF_k \right) .
\end{aligned}
\end{equation}
where the collection of the gradient tracking variables for the primal and dual decision variables are denoted as follows:
\begin{equation*}
\begin{aligned}
\mathbf{p}_k&:=\left[ p_{1,k},p_{2,k},\dots ,p_{n,k} \right] ^{\top}\in \mathbb{R} ^{n\times p},
\\
\mathbf{q}_k&:=\left[ q_{1,k},q_{2,k},\dots ,q_{n,k} \right] ^{\top}\in \mathbb{R} ^{n\times d}.
\end{aligned}
\end{equation*}

Furthermore, taking the average over all nodes on both sides of the algorithm yields the following equations:
\begin{equation}
\begin{aligned}
\bar{x}_{k+1}&:=\frac{\mathbf{1}^{\top}}{n}\mathbf{x}_{k+1}=\bar{x}_k-\gamma \frac{\mathbf{1}^{\top}}{n}\left( 2\nabla _xF_k-\nabla _xF_{k-1} \right) ,
\\
\bar{y}_{k+1}&:=\frac{\mathbf{1}^{\top}}{n}\mathbf{y}_{k+1}=\bar{y}_k+\gamma \frac{\mathbf{1}^{\top}}{n}\left( 2\nabla _yF_k-\nabla _yF_{k-1} \right),
\\
\bar{p}_{k+1}&:=\frac{\mathbf{1}^{\top}}{n}\mathbf{p}_{k+1}=\frac{\mathbf{1}^{\top}}{n}\nabla _xF_{k+1},
\\
\bar{q}_{k+1}&:=\frac{\mathbf{1}^{\top}}{n}\mathbf{q}_{k+1}=\frac{\mathbf{1}^{\top}}{n}\nabla _yF_{k+1}.
\end{aligned}
\end{equation}
It can be observed that, on average, DOGT employs a centralized OGDA update scheme, with the gradient tracking variables asymptotically aligning with the average gradient across all nodes. The consensus of these variables is ensured by the weighting matrix $W$, which satisfies Assumption \ref{Ass_graph}.

To further enhance the convergence of DOGT, we incorporate the accelerated gossip consensus protocol \cite{liu2011accelerated} and obtained ADOGT, leveraging prior knowledge of the spectrum gap $\rho_W$. Specifically, we replace the communication step of DOGT (cf. line 5 in Algorithm \ref{Alg_DOST}) with the following update, executed a finite number of $T$ times at each iteration:
\begin{equation}\label{Eq_acc_consensus}
\theta _{i,k+1}\gets \left( 1+\eta \right) \sum_{j\in \mathcal{N} _i}{w_{i,j}\theta _{j,k+1}}-\eta \theta _{i,k+1},
\end{equation}
where $\eta =\left( 1-\sqrt{1-\rho _W} \right) /\left( 1+\sqrt{1-\rho _W} \right)$. Then, we obtain a generated weight matrix $M_T$ of ADOGT as follows:
\begin{equation}\label{Eq_weight_matrix_M}
M_{t+1}=\left( 1+\eta \right) WM_t-\eta M_{t-1},
\end{equation}
where $t=0,1,\dots,T-1$, and $M_{-1}=M_0=\mathbf{I}$. We will show that ADOGT achieves the optimal convergence rate with this accelerated weight matrix (cf., Theorem \ref{Thm_acc_DOGT}).

\section{Convergence Analysis}\label{Sec_convergence}

In this section, we provide the convergence analysis of DOGT and ADOGT for SC-SC smooth objective functions. For simplicity, we further denote
\begin{equation*}
\begin{aligned}
\mathbf{z}_k:=\left[ \mathbf{x}_k, \mathbf{y}_k \right] ,\,
\mathbf{r}_k:=\left[ \mathbf{p}_k, -\mathbf{q}_k \right],\,
\nabla F_k:=\left[ \nabla _xF_k, -\nabla _yF_k \right].
\end{aligned}
\end{equation*}
Then, letting $\boldsymbol{\varepsilon }_k=\nabla F_{k+1}-\nabla F_k-\nabla F_k+\nabla F_{k-1}$, we get the update rule of $\bar{z}_k:=\mathbf{1}^{\top}\mathbf{z}_k/n$ as follows:
\begin{equation}
\bar{z}_{k+1}=\bar{z}_k-\gamma \frac{\mathbf{1}^{\top}}{n}\nabla F_{k+1} +\gamma \frac{\mathbf{1}^{\top}}{n}\boldsymbol{\varepsilon }_k,
\end{equation}
and
\begin{equation}\label{Eq_update_z_ave}
\begin{aligned}
&\bar{z}_{k+1}-\gamma \frac{\mathbf{1}^{\top}}{n}\left( \nabla F_{k+1}-\nabla F_k \right) 
\\
&=\bar{z}_k-\gamma \frac{\mathbf{1}^{\top}}{n}\left( \nabla F_k-\nabla F_{k-1} \right) -\gamma \frac{\mathbf{1}^{\top}}{n}\nabla F_{k+1}.
\end{aligned}
\end{equation}

\subsection{Linear Convergence}

To prove the convergence of DOGT, we define the following Lyapunov function:
\begin{equation}
\begin{aligned}
\varPsi _k&:=\left\| \varXi _k \right\| ^2+\frac{\gamma L}{n}\left\| \mathbf{z}_k-\mathbf{z}_{k-1} \right\| ^2
\\
&\quad+c_1\left\| \mathbf{z}_k-\mathbf{1}\bar{z}_k \right\| ^2+c_2\left\| \mathbf{r}_k-\mathbf{1}\bar{r}_k \right\| ^2,
\end{aligned}
\end{equation}
where $c_1$ and $c_2$ are coefficients to be designed, and the optimality gap to the optimal solution $z^*:=\left[ x^*,y^* \right]$ of problem \eqref{Prob_minimax_constained} is defined as
\begin{equation}\label{Eq_opt_gap}
\varXi _k:=\bar{z}_k-\gamma \frac{\mathbf{1}^{\top}}{n}\left( \nabla F_k-\nabla F_{k-1} \right) -z^*.
\end{equation}

Then, with the help of Lemmas \ref{Lem_z_iter_gap}-\ref{Lem_opt} in the appendix, the following theorem shows a linear convergence rate of DOGT.

\begin{Thm}\label{Thm_DOGT}
{Consider the DOGT algorithm as depicted in~\eqref{Eq_DOGT}.} Suppose Assumptions \ref{Ass_conv_conc}-\ref{Ass_graph} hold. Let the stepsize 
\begin{equation}\label{ss_theorem}
\gamma \leqslant \min \left\{ \frac{1}{64L}, \frac{\left( 1-\rho _W \right) ^2}{144L\sqrt{\rho _W}} \right\}.
\end{equation}
Then, we have for all $k \geqslant 0$,
\begin{equation}
\begin{aligned}
\varPsi _k\leqslant \left( 1-\min \left\{ \frac{3\gamma \mu}{4}, \frac{1-\rho _W}{8} \right\} \right) ^k\varPsi _0.
\end{aligned}
\end{equation}
\end{Thm}

\begin{proof}
By Lemma \ref{Lem_z_iter_gap} and Lemma \ref{Lem_opt}, we can obtain that
\begin{equation}
\begin{aligned}
&\left\| \varXi _{k+1} \right\| ^2+\frac{\gamma L}{n}\left\| \mathbf{z}_{k+1}-\mathbf{z}_k \right\| ^2
\\
&\leqslant \left( 1-\frac{3\gamma \mu}{4} \right) \left\| \varXi _k \right\| ^2+\frac{4\gamma ^3L^3}{n}\left\| \mathbf{z}_k-\mathbf{z}_{k-1} \right\| ^2
\\
&\quad +\frac{9\gamma L}{n}\left\| \mathbf{z}_k-\mathbf{1}\bar{z}_k \right\| ^2+\frac{9\gamma ^3L}{n\left( 1-\rho _W \right)}\left\| \mathbf{r}_k-\mathbf{1}\bar{r}_k \right\| ^2
\\
&\quad -\left( \frac{\gamma ^2}{4n}-\frac{8\gamma ^3L}{n} \right) \left\| \nabla F\left( \mathbf{1}\bar{z}_k \right) \right\| ^2.
\end{aligned}
\end{equation}
Letting $\gamma\leqslant 1/(8L)$ such that
\[
1-\frac{3\gamma \mu}{4}\geqslant 1-\frac{3\mu}{32L}\geqslant 4\gamma ^2L^2+4\gamma L,
\]
we can obtain 
\begin{equation}
\begin{aligned}
&\left\| \varXi _{k+1} \right\| ^2+\frac{\gamma L}{n}\left\| \mathbf{z}_{k+1}-\mathbf{z}_k \right\| ^2
\\
&\leqslant \left( 1-\frac{3\gamma \mu}{4} \right) \left( \left\| \varXi _k \right\| ^2+\left\| \mathbf{z}_k-\mathbf{z}_{k-1} \right\| ^2 \right) 
\\
&\quad+\frac{9\gamma L}{n}\left\| \mathbf{z}_k-\mathbf{1}\bar{z}_k \right\| ^2+\frac{9\gamma ^3L}{n\left( 1-\rho _W \right)}\left\| \mathbf{r}_k-\mathbf{1}\bar{r}_k \right\| ^2
\\
&\quad-\frac{4\gamma L}{n}\left\| \mathbf{z}_k-\mathbf{z}_{k-1} \right\| ^2-\left( \frac{\gamma ^2}{4n}-\frac{8\gamma ^3L}{n} \right) \left\| \nabla F\left( \mathbf{1}\bar{z}_k \right) \right\| ^2.
\end{aligned}
\end{equation}

Then, combining Lemma \ref{Lem_cons} and \ref{Lem_track} , we can obtain the contraction of the Lyapunov function:
\begin{equation}
\begin{aligned}
&\varPsi _{k+1}
\\
&\leqslant \left( 1-\min \left\{ \frac{3\gamma \mu}{4}, \frac{1-\rho _W}{8} \right\} \right) \varPsi _{k}
\\
&\quad+\left( \frac{8\gamma ^2L^4\rho _W}{1-\rho _W}c_2+\frac{4\gamma ^2\rho _WL^2}{1-\rho _W}c_1-\frac{4\gamma ^2L^2}{n} \right) \left\| \mathbf{z}_k-\mathbf{z}_{k-1} \right\| ^2
\\
&\quad+\left( \frac{9L^2\rho _W}{1-\rho _W}c_2+\frac{9\gamma L}{n}-\frac{1-\rho _W}{4}c_1 \right) \left\| \mathbf{z}_k-\mathbf{1}\bar{z}_k \right\| ^2
\\
&\quad+\left( \frac{4\gamma ^2\rho _W}{1-\rho _W}c_1+\frac{9\gamma ^3L}{n\left( 1-\rho _W \right)}-\frac{1-\rho _W}{8}c_2 \right) \left\| \mathbf{r}_k-\mathbf{1}\bar{r}_k \right\| ^2
\\
&\quad-\left( \frac{\gamma ^2}{4n}-\frac{8\gamma ^3L}{n}-c_2\frac{16\gamma ^2L^2\rho _W}{1-\rho _W} \right) \left\| \nabla F\left( \mathbf{1}\bar{z}_k \right) \right\| ^2.
\end{aligned}
\end{equation}
Set the parameters of the Lyapunov function as follows:
\begin{equation}
c_1=\frac{72\gamma L}{n\left( 1-\rho _W \right)},  \quad
c_2=\frac{4608\gamma ^3L}{n\left( 1-\rho _W \right) ^3}.
\end{equation}
Then, letting the stepsize further satisfy \eqref{ss_theorem} such that the coefficients of the last four terms on the right-hand side of the inequality are all less than or equal to 0, we complete the proof.
\end{proof}

\begin{Rem}
Theorem \ref{Thm_DOGT} shows that DOGT converges to the optimal solution of problem \eqref{Prob_minimax_constained} at a linear rate. And, by the upper bound of the stepsize \eqref{ss_theorem}, we can drive that it reaches a suboptimality level of $\epsilon > 0$ in at most $K$ iterations:
\begin{equation}\label{Eq_iter_complexity_DOGT}
K=\mathcal{O} \left( \left( \kappa \left( 1+\frac{\sqrt{\rho _W}}{\left( 1-\rho _W \right) ^2} \right) +\frac{1}{1-\rho _W} \right) \log \left( \epsilon ^{-1} \right) \right),
\end{equation}
where $\kappa:=L/\mu$ denotes the condition number of the overall objective function, $\mathcal{O} \left( \cdot \right) $ hides the constants. Compared to the GT-EG method \cite{mukherjee2020decentralized}, DOGT achieves at least the same convergence rate and improves it by a factor of $\mathcal{O} \left( \kappa ^{1/3} \right)$ when $\rho_W$ is small, indicating strong network connectivity. 
\end{Rem}

\subsection{Optimal Convergence Rate}
For the ADOGT algorithm with accelerated consensus protocol \eqref{Eq_acc_consensus}, the following theorem gives the optimal convergence rate and communication complexity matching the existing lower bound in deterministic settings \cite{beznosikov2020distributed}.

\begin{Thm}\label{Thm_acc_DOGT}
Suppose Assumptions \ref{Ass_conv_conc}-\ref{Ass_graph} hold. Let the stepsize satisfy \eqref{ss_theorem}, and the number of communication steps at each round $T=\lceil \ln \left( 2 \right) /\sqrt{1-\sqrt{\rho _W}} \rceil $. Then, for a given $\rho_W$, the ADOGT algorithm achieves a linear rate of $\mathcal{O} \left( \kappa \log \left( \epsilon ^{-1} \right) \right) $, and the number of communication rounds $R$ required to reach a suboptimality level of $\epsilon >0$ is 
\begin{equation}
R=\mathcal{O} \left( \frac{\kappa}{\sqrt{1-\sqrt{\rho _W}}}\log \left( \epsilon ^{-1} \right) \right).
\end{equation}
\end{Thm}

\begin{proof}
The proof of ADOGT follows a similar approach to that of DOGT, with the weight matrix $W$ replaced by $M_T$. In specific, after executing $T$ steps of the accelerated consensus \eqref{Eq_acc_consensus}, we obtain an equivalent weight matrix $M_T$ generated by \eqref{Eq_weight_matrix_M}.
Then, by Proposition 3 in \cite{liu2011accelerated}, we get
\begin{equation}
\rho _M:=\left\| M_T-\mathbf{J} \right\| ^2\leqslant 2\left( 1-\sqrt{1-\sqrt{\rho _W}} \right) ^{2T}.
\end{equation}
Letting the number of communication steps at each round $T=\lceil \ln(2)/\sqrt{1-\sqrt{\rho _W}} \rceil $, we have $1-\rho _M \geqslant 1/2$. 
Then, replacing $\rho_W$ in the iteration complexity of DOGT in \eqref{Eq_iter_complexity_DOGT} with $\rho_M$ and noticing that the total number of communication steps equals the iterations multiplied by $T$, we obtain the optimal linear convergence rate and communication complexity.
\end{proof}

\begin{figure*}[!t]
    \centering
    \subfloat[Trajectory]
    {
        \begin{minipage}[t]{0.32\textwidth}
            \centering
            \includegraphics[width=\textwidth]{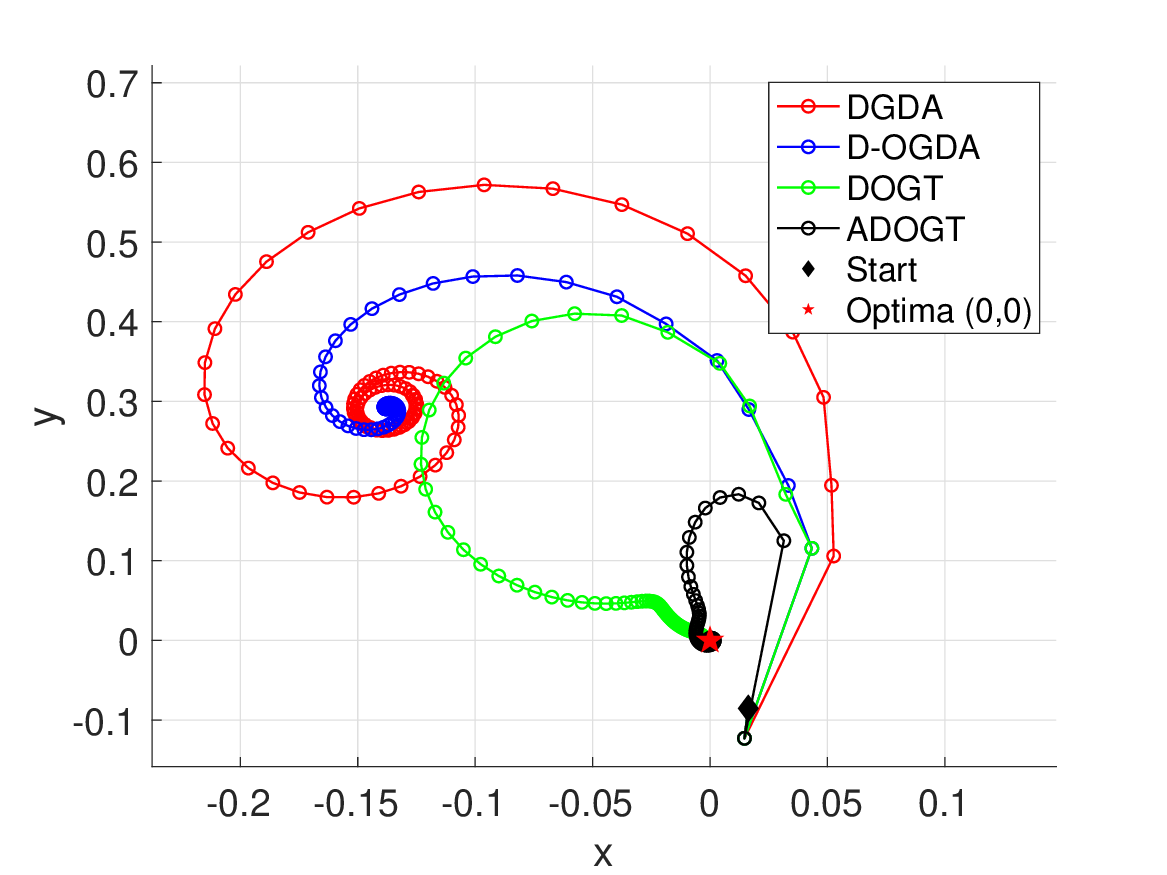}
        \end{minipage}
    }
    \subfloat[Residual]
    {
        \begin{minipage}[t]{0.32\textwidth}
            \centering 
            \includegraphics[width=\textwidth]{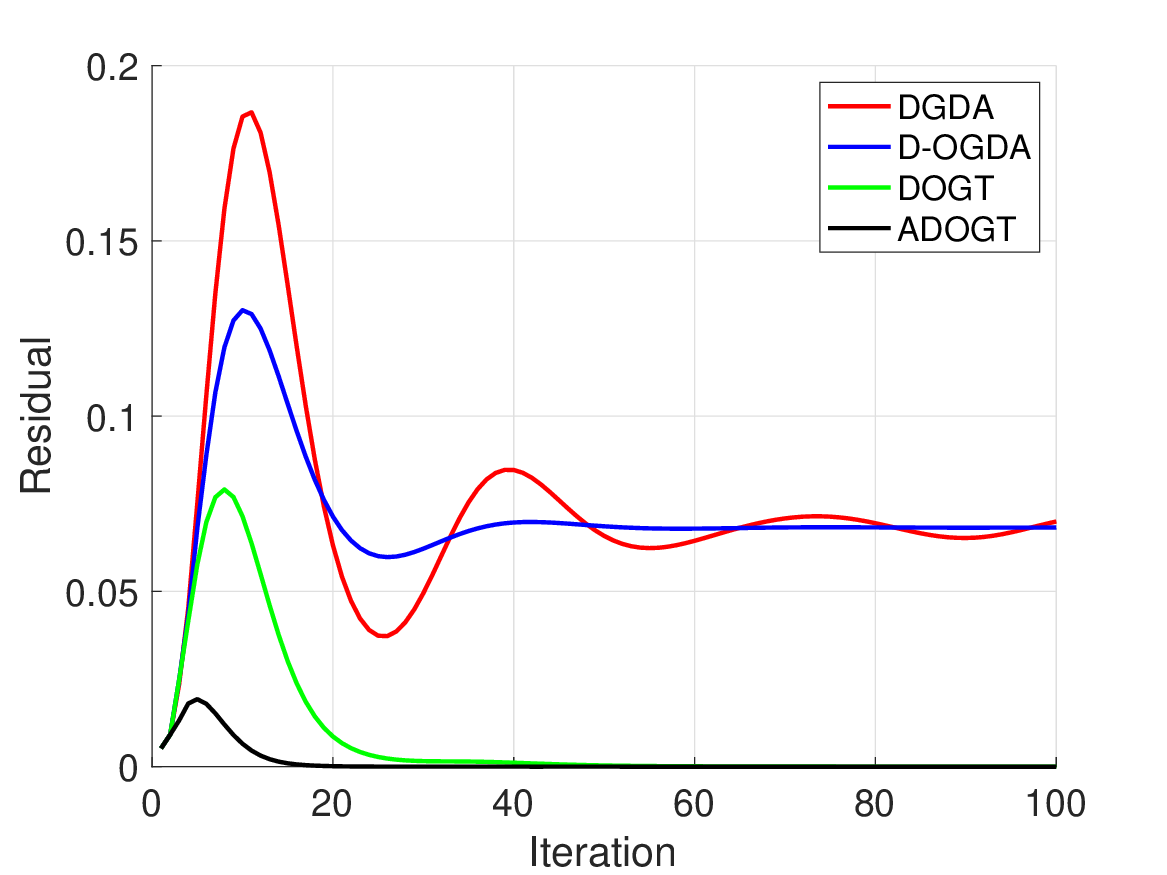} 
        \end{minipage}
    }
    \subfloat [Consensus error]
    {
        \begin{minipage}[t]{0.32\textwidth}
            \centering 
            \includegraphics[width=\textwidth]{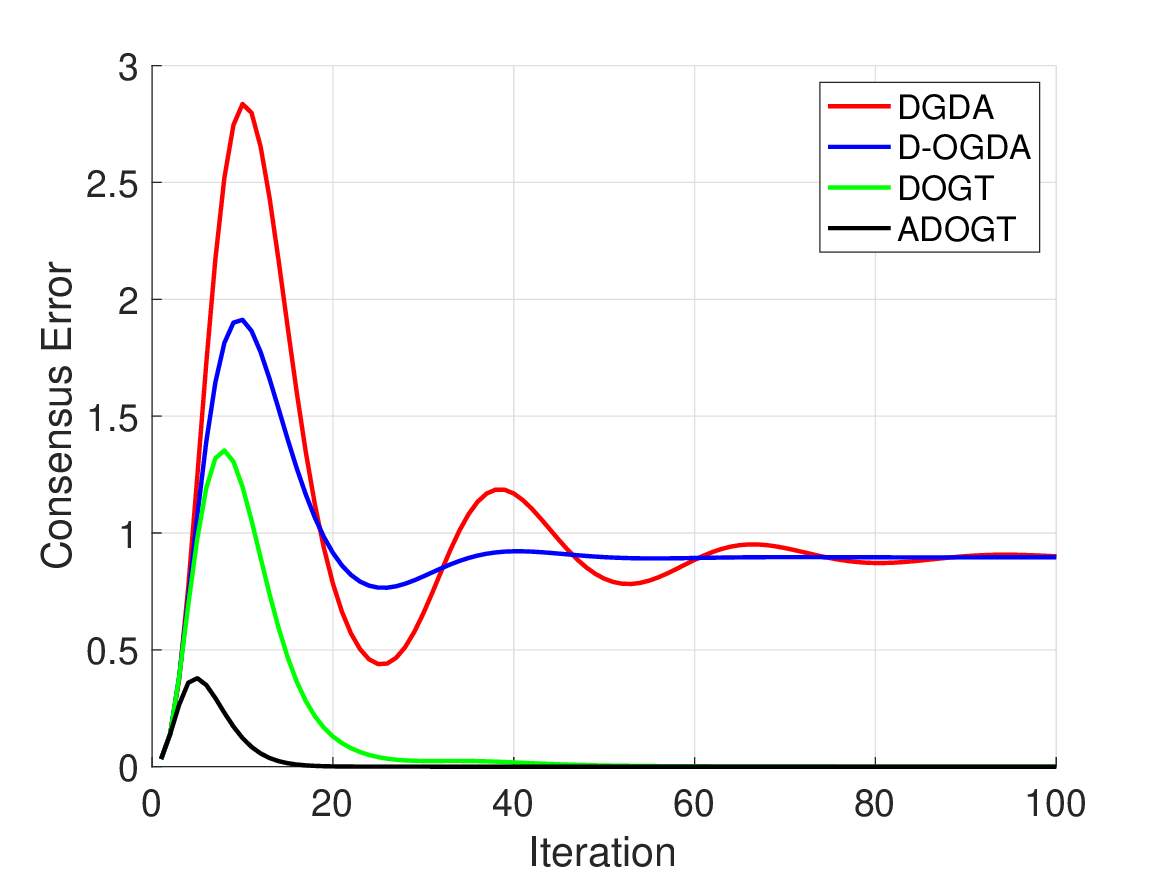} 
        \end{minipage}
    } 
    \caption{Comparison of the convergence performance between DGDA, D-OGDA, DOGT and ADOGT.}
    \label{Fig_comprison_OGT}
\end{figure*}

\section{Numerical Experiments}\label{Sec_experiment}

In this section, we conduct numerical experiments to compare the convergence performance between DGDA, D-OGDA, DOGT, and ADOGT. To this end, we consider the following synthetic example:
\begin{equation}
\begin{aligned}
f_i\left( x_i,y_i \right) =x_{i}^{T}y_i+\frac{\mu}{2}\left\| x_i-a_i \right\| ^2-\frac{\mu}{2}\left\| y_i-b_i \right\| ^2,
\end{aligned}
\end{equation}
where $i\in[n]$, $a_i\in \mathbb{R} ^p$ and $b_i\in \mathbb{R} ^d$ are different for each nodes. We consider $n=16$ nodes connected as an undirected ring graph. We set the dimensions $p=d=2$, the regularization constant $\mu=0.1$, and \hy{the stepsize $\gamma =0.1$}, respectively. For ADOGT, we set the number of communication steps at each round $T=4$ as suggested by Theorem \ref{Thm_acc_DOGT}. Moreover, we set the optimal solution $z^*=\left[ 0, 0; 0, 0 \right]$ by choosing $\sum\nolimits_{i=1}^n{a_i}=\sum\nolimits_{i=1}^n{b_i}=[0;0]$.

In Fig. \ref{Fig_comprison_OGT}, we plot the trajectory with respect to the primal and dual decision variables, the residual $1/n\left\| \mathbf{z}_k-\mathbf{1}z^* \right\| ^2$, and the consensus error $1/n\left\| \mathbf{z}_k-\mathbf{1}\bar{z} \right\|$ for each algorithm, respectively. 
The trajectories show that DOGT and ADOGT converge to the optimal solution, with ADOGT achieving faster convergence due to the integration of an accelerated consensus protocol. This demonstrates the robustness of GT-based methods against function heterogeneity. Instead, D-OGDA converges to a non-optimal point, and DGDA exhibits periodic oscillations without achieving convergence. Their non-zero residuals further support this behavior. 
Additionally, we observe that the consensus errors of DGDA and D-OGDA do not converge to zero. This is due to the heterogeneity of the objective functions, as also observed in \cite{cai2024diffusion}, which leads to inconsistent optimal points across nodes. In contrast, DOGT and ADOGT remain robust to this issue. These results demonstrate the effectiveness of the proposed algorithms.

\section{Conclusion}\label{Sec_conclusion}
In this paper, we have proposed a distributed optimistic gradient tracking method, DOGT, along with its accelerated variant, ADOGT, for solving distributed minimax optimization problems over networks. We have also provided rigorous theoretical analysis to show that DOGT achieves a linear convergence rate to the optimal solution for SC-SC objective functions. Furthermore, the proposed ADOGT integrated with accelerated consensus protocol achieves an optimal convergence rate and communication complexity that matches the existing lower bound. Numerical experiments demonstrate the effectiveness of the proposed algorithms. \hy{Future work will focus on the design and analysis of algorithms in stochastic and non-convex settings, as well as exploring their potential applications in the distributed training of GANs.}

\appendix
\section{Supporting Lemmas}
In this section, we introduce several supporting lemmas for the proof of the main results, highlighting important contraction properties of the consensus error, gradient tracking error, and optimality gap.

\begin{Lem}\label{Lem_z_iter_gap}
Suppose Assumptions \ref{Ass_smoothness} and \ref{Ass_graph} hold. Then, we have for all $ k \geqslant 0$,
\begin{equation}\label{Eq_z_ietr_gap}
\begin{aligned}
&\left\| \mathbf{z}_{k+1}-\mathbf{z}_k \right\| ^2
\\
&\leqslant 4\gamma ^2L^2\left\| \mathbf{z}_k-\mathbf{z}_{k-1} \right\| ^2+\left( 4+8\gamma ^2L^2 \right) \left\| \mathbf{z}_k-\mathbf{1}\bar{z}_k \right\| ^2
\\
 &\quad+8\gamma ^2\left\| \mathbf{r}_k-\mathbf{1}\bar{r}_k \right\| ^2+8n\gamma ^2\left\| \nabla f\left( \bar{z}_k \right) \right\| ^2.
\end{aligned}
\end{equation}
\end{Lem}

\begin{proof}
By the updates of DOGT \eqref{Eq_DOGT} and \eqref{Eq_update_z_ave}, we get
\begin{equation*}
\begin{aligned}
&\left\| \mathbf{z}_{k+1}-\mathbf{z}_k \right\| ^2
\\
&\leqslant 2\left\| \left( W-\mathbf{I} \right) \left( \mathbf{z}_k-\mathbf{1}\bar{z}_k \right) \right\| ^2
\\
&\quad +2\gamma ^2\left\| W\left( \left( \mathbf{r}_k+\nabla F_k-\nabla F_{k-1} \right) \right) \right\| ^2
\\
&\leqslant 4\left\| \mathbf{z}_k-\mathbf{1}\bar{z}_k \right\| ^2+4\gamma ^2L^2\left\| \mathbf{z}_k-\mathbf{z}_{k-1} \right\| ^2
\\
&\quad +4\gamma ^2\left\| W\left( \mathbf{r}_k-\mathbf{1}\bar{r}_k \right) +\mathbf{1}\bar{r}_k \right\| ^2,
\end{aligned}
\end{equation*}
where we used the fact $\left\| W-\mathbf{I} \right\| ^2\leqslant 2$ and the smoothness of $f_i$.
Then, noticing that $\bar{r}_k=\frac{\mathbf{1}^{\top}}{n}\nabla F_{k+1}$, we can obtain the result in \eqref{Eq_z_ietr_gap}.
\end{proof}

The following lemma shows the contraction of the consensus error.

\begin{Lem}[Consensus error]\label{Lem_cons}
Suppose Assumptions \ref{Ass_smoothness} and \ref{Ass_graph} hold. Let the stepsize satisfy $\gamma \leqslant 1/\left( 4L \right) $, we get $\forall k \geqslant 0$,
\begin{equation}\label{Eq_consensus}
\begin{aligned}
&\left\| \mathbf{z}_{k+1}-\mathbf{1}\bar{z}_{k+1} \right\| ^2
\\
&\leqslant \frac{1+\rho _W}{2}\left\| \mathbf{z}_k-\mathbf{1}\bar{z}_k \right\| ^2+\frac{2\gamma ^2\left( 1+\rho _W \right) \rho _W}{1-\rho _W}\left\| \mathbf{r}_k-\mathbf{1}\bar{r}_k \right\| ^2
\\
&\quad+\frac{2\gamma ^2\left( 1+\rho _W \right) \rho _WL^2}{1-\rho _W}\left\| \mathbf{z}_k-\mathbf{z}_{k-1} \right\| ^2.
\end{aligned}
\end{equation}
\end{Lem}

\begin{proof}
By the update rule of DOGT \eqref{Eq_DOGT} and Young's inequality, we get
\begin{equation*}
\begin{aligned}
&\left\| \mathbf{z}_{k+1}-\mathbf{1}\bar{z}_{k+1} \right\| ^2
\\
&=\left\| \left( \mathbf{I}-\mathbf{J} \right) W\left( \mathbf{z}_k-\gamma \left( \mathbf{r}_k+\nabla F_k-\nabla F_{k-1} \right) \right) \right\| ^2
\\
&\leqslant \left( 1+\lambda \right) \left\| \left( W-\mathbf{J} \right) \left( \mathbf{z}_k-\mathbf{1}\bar{z}_k \right) \right\| ^2
\\
&\quad +2\gamma ^2\left( 1+\lambda ^{-1} \right) \left\| \left( W-\mathbf{J} \right) \left( \mathbf{r}_k-\mathbf{1}\bar{r}_k \right) \right\| ^2
\\
&\quad+2\gamma ^2\left( 1+\lambda ^{-1} \right) \left\| \left( W-\mathbf{J} \right) \left( \nabla F_k-\nabla F_{k-1} \right) \right\| ^2.
\end{aligned}
\end{equation*}
Then, letting $\lambda =\frac{1-\rho _W}{2\rho _W}$, and using the smoothness of $f_i$ and Assumption \ref{Ass_graph}, we complete the proof.
\end{proof}

Next, we have the following lemma showing the contraction of the gradient tracking error.

\begin{Lem}[Gradient tracking error]\label{Lem_track}
Suppose Assumptions \ref{Ass_smoothness} and \ref{Ass_graph} hold. Let the stepsize satisfy
\begin{equation}\label{ss_tracking}
\gamma \leqslant \left\{ \frac{1}{4L}, \frac{1-\rho _W}{8L\sqrt{\rho _W}} \right\}.
\end{equation}
Then, we get for all $k \geqslant 0$,
\begin{equation}
\begin{aligned}
&\left\| \mathbf{r}_{k+1}-\mathbf{1}\bar{r}_{k+1} \right\| ^2
\\
&\leqslant \frac{3+\rho _W}{4}\left\| \mathbf{r}_k-\mathbf{1}\bar{r}_k \right\| ^2+\frac{8\gamma ^2L^4\rho _W}{1-\rho _W}\left\| \mathbf{z}_k-\mathbf{z}_{k-1} \right\| ^2
\\
&\quad +\frac{9L^2\rho _W}{1-\rho _W}\left\| \mathbf{z}_k-\mathbf{1}\bar{z}_k \right\| ^2+\frac{16n\gamma ^2L^2\rho _W}{1-\rho _W}\left\| \nabla f\left( \bar{z}_k \right) \right\| ^2.
\end{aligned}
\end{equation}
\end{Lem}

\begin{proof}
According to the updates of DOGT \eqref{Eq_DOGT}, we have
\begin{equation*}
\begin{aligned}
&\left\| \mathbf{r}_{k+1}-\mathbf{1}\bar{r}_{k+1} \right\| ^2
\\
&=\left\| \left( W-\mathbf{J} \right) \left( \mathbf{r}_k+\nabla F_{k+1}-\nabla F_k \right) \right\| ^2
\\
&\leqslant \left( 1+\lambda \right) \left\| \left( W-\mathbf{J} \right) \left( \mathbf{r}_k-\mathbf{1}\bar{r}_k \right) \right\| ^2
\\
&\quad+\left( 1+\lambda ^{-1} \right) \left\| \left( W-\mathbf{J} \right) \left( \nabla F_{k+1}-\nabla F_k \right) \right\| ^2
\\
&\leqslant \frac{1+\rho _W}{2}\left\| \mathbf{r}_k-\mathbf{1}\bar{r}_k \right\| ^2+\frac{\left( 1+\rho _W \right) \rho _WL^2}{1-\rho _W}\left\| \mathbf{z}_{k+1}-\mathbf{z}_k \right\| ^2,
\end{aligned}
\end{equation*}
where we used Young's inequality with $\lambda =\frac{1-\rho _W}{2\rho _W}$ and the smoothness of $f_i$. Then, 
letting the stepsize satisfy \eqref{ss_tracking} and applying Lemma \ref{Lem_z_iter_gap}, we complete the proof.
\end{proof}

We are now in a position to show the key lemma on the contraction of the optimality gap defined in \eqref{Eq_opt_gap}.

\begin{Lem}[Optimality gap]\label{Lem_opt}
Suppose Assumptions \ref{Ass_conv_conc}-\ref{Ass_graph} hold. Let the stepsize satisfy 
\begin{equation}\label{Eq_stepsize_opt}
\gamma \leqslant \left\{ \frac{1}{8L}, \frac{1-\rho _W}{8L\rho _W} \right\}.
\end{equation}
Then, we have for all $k\geqslant0$,
\begin{equation}
\begin{aligned}
&\left\| \varXi _{k+1} \right\| ^2
\\
&\leqslant \left( 1-\frac{3\gamma \mu}{4} \right) \left\| \varXi _k \right\| ^2+\frac{5\gamma ^2L^2}{4n}\left\| \mathbf{z}_k-\mathbf{z}_{k-1} \right\| ^2
\\
&\quad +\frac{4\gamma L}{n}\left\| \mathbf{z}_k-\mathbf{1}\bar{z}_k \right\| ^2+\frac{9\gamma ^3L\rho _W}{n\left( 1-\rho _W \right)}\left\| \mathbf{r}_k-\mathbf{1}\bar{r}_k \right\| ^2
\\
&\quad -\frac{\gamma ^2}{4n}\left\| \nabla F\left( \mathbf{1}\bar{z}_k \right) \right\| ^2.
\end{aligned}
\end{equation}
\end{Lem}

\begin{proof}
According to \eqref{Eq_update_z_ave}, we have
\begin{equation*}
\begin{aligned}
&\left\| \varXi _{k+1} \right\| ^2
=\left\| \varXi _k-\gamma \frac{\mathbf{1}^{\top}}{n}\nabla F_{k+1} \right\| ^2
\\
&=\left\| \varXi _k \right\| ^2+\left\| \gamma \frac{\mathbf{1}^{\top}}{n}\nabla F_{k+1} \right\| ^2-2\left< \gamma \frac{\mathbf{1}^{\top}}{n}\nabla F_{k+1}, \varXi _k \right> .
\end{aligned}
\end{equation*}
Since $\varXi _{k+1}=\varXi _k-\gamma \frac{\mathbf{1}^{\top}}{n}\nabla F_{k+1}$, we have
\begin{equation*}
\begin{aligned}
&\left\| \varXi _{k+1} \right\| ^2
\\
&=\left\| \varXi _k \right\| ^2+\left\| \gamma \frac{\mathbf{1}^{\top}}{n}\nabla F_{k+1} \right\| ^2
\\
&\quad-2\left< \gamma \frac{\mathbf{1}^{\top}}{n}\nabla F_{k+1}, \varXi _{k+1}+\gamma \frac{\mathbf{1}^{\top}}{n}\nabla F_{k+1} \right> 
\\
&=\left\| \varXi _k \right\| ^2+\frac{\gamma ^2}{n}\left\| \nabla F_{k+1} \right\| ^2-2\left< \gamma \frac{\mathbf{1}^{\top}}{n}\nabla F_{k+1},\bar{z}_{k+1}-z^* \right> 
\\
&\quad-2\left< \gamma \frac{\mathbf{1}^{\top}}{n}\nabla F_{k+1},\gamma \frac{\mathbf{1}^{\top}}{n}\nabla F_k \right>.
\end{aligned}
\end{equation*}
Noticing that $a^2-2ab=\left( a-b \right) ^2-b^2, \forall a,b\in \mathbb{R} $, we get

\begin{equation}\label{Eq_Desc_1}
\begin{aligned}
\left\| \varXi _{k+1} \right\| ^2
&=\left\| \varXi _k \right\| ^2-2\left< \gamma \frac{\mathbf{1}^{\top}}{n}\nabla F_{k+1},\bar{z}_{k+1}-z^* \right> 
\\
&\quad  +\frac{\gamma ^2}{n}\left\| \nabla F_{k+1}-\nabla F_k \right\| ^2-\frac{\gamma ^2}{n}\left\| \nabla F_k \right\| ^2.
\end{aligned}
\end{equation}

For the inner product term, we have
\begin{equation*}
\begin{aligned}
&\left< \gamma \frac{\mathbf{1}^{\top}}{n}\nabla F\left( \mathbf{z}_{k+1} \right) ,\bar{z}_{k+1}-z^* \right> 
\\
&=\frac{\gamma}{n}\sum_{i=1}^n{\left< \nabla f_i\left( z_{i,k+1} \right) ,z_{i,k+1}-z^*-\left( z_{i,k+1}-\bar{z}_{k+1} \right) \right>}.
\end{aligned}
\end{equation*}
Then, by Assumptions \ref{Ass_conv_conc} and \ref{Ass_smoothness}, we have
\begin{equation*}
\begin{aligned}
&\left< \gamma \frac{\mathbf{1}^{\top}}{n}\nabla F\left( \mathbf{z}_{k+1} \right) ,\bar{z}_{k+1}-z^* \right> 
\\
&\geqslant \frac{\gamma}{n}\sum_{i=1}^n{\left( f_i\left( z_{i,k+1} \right) -f_i\left( z^* \right) +\frac{\mu}{2}\left\| z_{i,k+1}-z^* \right\| ^2 \right)}
\\
&\quad-\frac{\gamma}{n}\sum_{i=1}^n{\left( f_i\left( z_{i,k+1} \right) -f_i\left( \bar{z}_{k+1} \right) +\frac{L}{2}\left\| z_{i,k+1}-\bar{z}_{k+1} \right\| ^2 \right)}
\\
&=\gamma f\left( \left( \bar{z}_{k+1} \right) -f\left( z^* \right) \right) +\frac{\gamma \mu}{2n}\left\| \mathbf{z}_{k+1}-\mathbf{1}z^* \right\| ^2
\\
&\quad-\frac{\gamma L}{2n}\left\| \mathbf{z}_{k+1}-\mathbf{1}\bar{z}_{k+1} \right\| ^2.
\end{aligned}
\end{equation*}

Putting it back to \eqref{Eq_Desc_1} and using $f\left( \bar{z}_{k+1} \right) -f\left( z^* \right) \geqslant \frac{\mu}{2}\left\| \bar{z}_{k+1}-z^* \right\| ^2$, we get
\begin{equation*}
\begin{aligned}
\left\| \varXi _{k+1} \right\| ^2
&\leqslant \left\| \varXi _k \right\| ^2+\frac{\gamma L}{n}\left\| \mathbf{z}_{k+1}-\mathbf{1}\bar{z}_{k+1} \right\| ^2
\\
&\quad-\gamma \mu \left\| \bar{z}_{k+1}-z^* \right\| ^2-\frac{\gamma \mu}{n}\left\| \mathbf{z}_{k+1}-\mathbf{1}z^* \right\| ^2
\\
&\quad+\frac{\gamma ^2L^2}{n}\left\| \mathbf{z}_{k+1}-\mathbf{z}_k \right\| ^2-\frac{\gamma ^2}{n}\left\| \nabla F_k \right\| ^2.
\end{aligned}
\end{equation*}

Noticing that 
\begin{equation*}
\begin{aligned}
&-\frac{1}{n}\left\| \mathbf{z}_{k+1}-\mathbf{1}z^* \right\| ^2
\\
&=-\frac{1}{n}\left\| \mathbf{z}_{k+1}-\mathbf{1}\bar{z}_{k+1} \right\| ^2-\left\| \bar{z}_{k+1}-z^* \right\| ^2
\\
&\quad -\frac{2}{n}\left< \mathbf{z}_{k+1}-\mathbf{1}\bar{z}_{k+1},\mathbf{1}\bar{z}_{k+1}-\mathbf{1}z^* \right> 
\\
&\leqslant \frac{1}{n}\left\| \mathbf{z}_{k+1}-\mathbf{1}\bar{z}_{k+1} \right\| ^2-\frac{1}{2}\left\| \bar{z}_{k+1}-z^* \right\| ^2,
\end{aligned}
\end{equation*}
and
\begin{equation*}
\begin{aligned}
-\left\| \bar{z}_{k+1}-z^* \right\| ^2
&=-\left\| \varXi _k-\gamma \frac{\mathbf{1}^{\top}}{n}\nabla F\left( \mathbf{z}_k \right) \right\| ^2
\\
&\leqslant -\frac{1}{2}\left\| \varXi _k \right\| ^2+\frac{\gamma ^2}{n}\left\| \nabla F_k \right\| ^2,
\end{aligned}
\end{equation*}
we get
\begin{equation*}
\begin{aligned}
&\left\| \varXi _{k+1} \right\| ^2
\\
&\leqslant \left( 1-\frac{3\gamma \mu}{4} \right) \left\| \varXi _k \right\| ^2+\frac{\gamma \left( L+\mu \right)}{n}\left\| \mathbf{z}_{k+1}-\mathbf{1}\bar{z}_{k+1} \right\| ^2
\\
&\quad+\frac{\gamma ^2L^2}{n}\left\| \mathbf{z}_{k+1}-\mathbf{z}_k \right\| ^2-\left( \frac{\gamma ^2}{n}-\frac{3\gamma ^3\mu}{2n} \right) \left\| \nabla F_k \right\| ^2.
\end{aligned}
\end{equation*}

Then, according to the obtained inequalities \eqref{Eq_z_ietr_gap} and \eqref{Eq_consensus} and letting the stepsize satisfy \eqref{Eq_stepsize_opt}, we can obtain that
\begin{equation*}
\begin{aligned}
&\left\| \varXi _{k+1} \right\| ^2
\\
&\leqslant \left( 1-\frac{3\gamma \mu}{4} \right) \left\| \varXi _k \right\| ^2+\frac{5\gamma ^2L^2}{4n}\left\| \mathbf{z}_k-\mathbf{z}_{k-1} \right\| ^2
\\
&\quad+\left( \frac{2\gamma L}{n}+\frac{9\gamma ^2L^2}{2n} \right) \left\| \mathbf{z}_k-\mathbf{1}\bar{z}_k \right\| ^2
\\
&\quad+\frac{9\gamma ^3L\rho _W}{n\left( 1-\rho _W \right)}\left\| \mathbf{r}_k-\mathbf{1}\bar{r}_k \right\| ^2
\\
&\quad+8\gamma ^4L^2\left\| \nabla f\left( \bar{z}_k \right) \right\| ^2-\left( \frac{\gamma ^2}{n}-\frac{3\gamma ^3\mu}{2n} \right) \left\| \nabla F_k \right\| ^2.
\end{aligned}
\end{equation*}
Then, noticing that $\left\| \nabla f\left( \bar{z}_k \right) \right\| ^2\leqslant \frac{1}{n}\left\| \nabla F\left( \mathbf{1}\bar{z}_k \right) \right\| ^2$ and
$-\left\| \nabla F_k \right\| ^2\leqslant L^2\left\| \mathbf{z}_k-\mathbf{1}\bar{z}_k \right\| ^2-\frac{1}{2}\left\| \nabla F\left( \mathbf{1}\bar{z}_k \right) \right\| ^2$, we complete the proof.
\end{proof}

\bibliographystyle{ieeetr}
\bibliography{reference}
\end{document}